\documentclass{amsart}

\usepackage{amsfonts}
\usepackage{amsthm}
\usepackage{amssymb}
\usepackage{mathrsfs}
\usepackage{graphicx}

\newtheorem{teor}{Theorem}
\newtheorem{lemma}[teor]{Lemma}
\newtheorem{conj}[teor]{Conjecture}
\newtheorem{prop}[teor]{Proposition}

\newtheorem{cor}[teor]{Corollary}

\newtheorem{question}[teor]{Question}

\begin{document}

\title{Inequalities detecting structural properties of a finite group}
\date{}
\author{Martino Garonzi \and Massimiliano Patassini}

\begin{abstract}
We prove several results detecting cyclicity or nilpotency of a finite group $G$ in terms of inequalities involving the orders of the elements of $G$ and the orders of the elements of the cyclic group of order $|G|$. We prove that, among the groups of the same order, the number of cyclic subgroups is minimal for the cyclic group and the product of the orders of the elements is maximal for the cyclic group.
\end{abstract}

\maketitle

\section{Introduction}

In this paper all groups are assumed to be finite.

The problem of detecting structural properties of a finite group by looking at element orders has been considered by various authors. Amiri, Jafarian Amiri and Isaacs in \cite{isaacs} proved that the sum of element orders of a finite group $G$ of order $n$ is maximal in the cyclic group of order $n$. The problem of minimizing sums of the form $\sum_{x \in G} o(x)^{-m}$, where $m$ is a positive integer and $o(x)$ denotes the order of $x$, was considered in \cite{suminverses}, however there is a mistake in the proof pointed out by Isaacs in \cite{moverf}. The main point of the argument in \cite{suminverses} is a pointwise argument, and the strong evidence that it is true suggests to state it as a conjecture.
\begin{conj} \label{mainconj}
Let $G$ be a finite group of order $n$ and let $C_n$ denote the cyclic group of order $n$. There exists a bijection $f : G \to C_n$ such that $o(x)$ divides $o(f(x))$ for all $x \in G$.
\end{conj}
This is proved in \cite{oib} by Frieder Ladisch in the case in which $G$ is solvable. Note that the existence of a bijection as in the conjecture is equivalent to the existence of a family $\{S_d\ :\ d|n\}$ of subsets of $G$ with the following properties (here $\varphi$ denotes Euler's totient function):
\begin{itemize}
\item The sets $S_d$ are pairwise disjoint and $G = \bigcup_{d|n} S_d$.
\item $x^d = 1$ for all $x \in S_d$, for all $d|n$.
\item $|S_d| = \varphi(d)$ for all $d|n$.
\end{itemize}

Indeed, given a bijection $f$ as in the conjecture, define $S_d$ to be the preimage via $f$ of the set of elements of $C_n$ of order $d$, and given a partition as above, define $f$ piecewise sending $S_d$ to the set of elements of $C_n$ of order $d$.

\ 

The existence of such a partition is claimed in \cite{suminverses} with a wrong proof, although this is not the main result of that paper. The main result of \cite{suminverses}, dealing with the sum $\sum_{x \in G} 1/o(x)^m$, is a consequence of our main result (Theorem \ref{mainth}(1) for $s < 0 = r$). Although in this paper we do not prove Conjecture \ref{mainconj}, such conjecture is worth mentioning because it is very much related to our results.

\ 

Let $\varphi$ denote Euler's totient function, i.e. $\varphi(n)$ denotes the number of integers in $\{1,\ldots,n\}$ coprime to $n$. In this paper we consider the sum $$\sum_{x \in G} \frac{o(x)^s}{\varphi(o(x))^r}$$ for $r,s$ real numbers and compare it with the case of the cyclic group of size $|G|$. In the case $s=1$, $r=0$ this sum equals the sum of element orders, in the case $s=r=1$ it equals the sum of the cyclic subgroup sizes. Moreover if $s < 0 = r$ we get an extension of the case considered in \cite{suminverses} and the case $s=0$, $r=1$ gives the number of cyclic subgroups. This last case was what motivated us in the beginning, and as a particular case of our main theorem we obtain the following. Let $d(n)$ denote the number of positive divisors of the integer $n$.

\begin{teor} \label{motiv}
Let $G$ be a finite group. Then $G$ has at least $d(|G|)$ cyclic subgroups and $G$ has exactly $d(|G|)$ cyclic subgroups if and only if $G$ is cyclic.
\end{teor}

This theorem follows from Corollary \ref{cormotiv}.

Using the same techniques we also prove, in Section \ref{sprod}, the following:

\begin{teor} \label{thprod}
Let $G$ be a finite group of order $n$ and let $P_G := \prod_{x \in G} o(x)$. Then $P_G \leq P_{C_n}$ with equality if and only if $G$ is cyclic.
\end{teor}

We also obtain a very interesting characterization of nilpotency (Theorem \ref{mainth}(2)):

\begin{teor}
Let $r < 0$ be a real number and let $G$ be a finite group of order $n$. Then $$\sum_{x \in G} {\left( \frac{o(x)}{\varphi(o(x))} \right)}^r \geq \sum_{x \in C_n} {\left( \frac{o(x)}{\varphi(o(x))} \right)}^r$$ and equality holds if and only if $G$ is nilpotent.
\end{teor}

Let us be more specific about what we actually do in the paper. We prove the following result.

\begin{teor} \label{mainth}
Let $r,s$ be two real numbers, let $G$ be a finite group of order divisible by $n$ and let $$R_{G,n}(r,s) := \sum_{x \in G, o(x)|n} \frac{o(x)^s}{\varphi(o(x))^r}, \hspace{1cm} T_{G,n}(r,s) := R_{G,n}(r,s) - R_{C_{|G|},n}(r,s).$$Set $R_G(r,s) := R_{G,|G|}(r,s)$ and $T_G(r,s) := T_{G,|G|}(r,s)$.
\begin{enumerate}
\item If $s < r$ and $s \leq 0$ then $T_{G,n}(r,s) \geq 0$ with equality if and only if $G$ contains a unique cyclic subgroup of order $m$, for every divisor $m$ of $n$.
\item If $s = r < 0$ then $T_{G,n}(r,s) \geq 0$ with equality if and only if $G$ contains a unique subgroup of order $n$ and such subgroup is nilpotent.
\item If $r \leq s-1$ and $s \geq 1$ then $T_{G}(r,s) \leq 0$ with equality if and only if $G$ is cyclic.
\item If $G$ is nilpotent and non-cyclic then the sign of $T_{G}(r,s)$ equals the sign of $r-s$.
\end{enumerate}
\end{teor}

We prove this in Section \ref{main}. For the case $s \leq \min \{0,r\}$ we use Lemma \ref{nov} (a combinatorial tool, which is a key result in this paper) and for the case $s \geq \max \{1,r+1\}$ we adapt the arguments of \cite{isaacs}. In section \ref{examples}, for any positive integer $\gamma$, we construct infinitely many finite groups $G$ with exactly $d(|G|)+\gamma$ cyclic subgroups.

\section{The main result} \label{main}

In this section we prove Theorem \ref{mainth}.

\ 

As usual $\mathbb{N}$ denotes the set of natural numbers (in particular $0 \not \in \mathbb{N}$).

Denote by $\mu: \mathbb{N} \to \{0,1,-1\}$ (the M{\"o}bius function) the map taking $n$ to $0$ if $n$ is divisible by a square different from $1$, to $1$ if $n$ is a product of an even number of distinct primes and to $-1$ if $n$ is a product of an odd number of distinct primes. The following result is well-known.

\begin{prop}[M{\"o}bius inversion formula]
Let $f,g: \mathbb{N} \to \mathbb{C}$ be two functions such that $g(n) = \sum_{d|n} f(d)$ for all $n \in \mathbb{N}$. Then $f(n) = \sum_{d|n} \mu(d) g(n/d)$ for all $n \in \mathbb{N}$.
\end{prop}

An important example is the following. It is well-known that$\sum_{d|n} \varphi(d) = n$ for all $n \in \mathbb{N}$. This is because in the cyclic group of order $n$ for any divisor $d$ of $n$ there are exactly $\varphi(d)$ elements of order $d$. Applying the M{\"o}bius inversion formula we obtain $\varphi(n) = \sum_{d|n} \mu(n/d) d$.

The following is our key combinatorial tool.

\begin{lemma} \label{nov}
Let $A,B: \mathbb{N} \to \mathbb{C}$ be two functions such that $$\sum_{m|n} A(m) = B(n) \hspace{1cm} \forall n \in \mathbb{N}.$$ For $j,m \in \mathbb{N}$ and $r,s$ two real numbers such that $s \leq \min \{0,r\}$ set $$g^{r,s}_{m,j} := \sum_{i|j} \frac{\mu(i) (mi)^s}{\varphi(mi)^r}.$$ Then we have:
\begin{enumerate}
\item Write the prime factorizations of $j$ and $m$ as $j = {p_1}^{t_1} \cdots {p_k}^{t_k}$ and $m = m' {p_1}^{\alpha_1} \cdots {p_l}^{\alpha_l}$ where $l \leq k$ and $(m',j)=1$. Then $$g_{m,j}^{r,s} = \frac{m^s}{\varphi(m)^r} \cdot \prod_{t=1}^l \left( 1 - p_t^{s-r} \right) \cdot \prod_{t=l+1}^k \left(1-\frac{p_t^s}{(p_t-1)^r} \right).$$In particular $g_{m,j}^{r,s} \geq 0$ always and $g_{m,j}^{r,s} = 0$ if and only if one of the following holds:

\begin{itemize}
\item $s = r = 0$ and $j \neq 1$.
\item $s = r \neq 0$ and $l \geq 1$, i.e. $j$ and $m$ are not coprime.
\item $s = 0 \neq r$ and $p_t = 2$ for some $t \in \{l+1,\ldots,k\}$, i.e. $j$ is even and $m$ is odd.
\end{itemize}

\item $\sum_{m|n} \frac{m^s}{\varphi(m)^r} A(m) = \sum_{k|n} g_{k,n/k}^{r,s} B(k)$.

\item Suppose $\alpha, \beta$ are positive functions $\mathbb{N} \to \mathbb{R}$ with $\beta(n) = \sum_{m|n} \alpha(m)$ and $B(k) \geq \beta(k)$ for all divisors $k$ of $n$. Then $$\sum_{m|n} \frac{m^s}{\varphi(m)^r} A(m) \geq \sum_{m|n} \frac{m^s}{\varphi(m)^r} \alpha(m).$$Moreover equality holds if and only if $B(k) = \beta(k)$ for all divisors $k$ of $n$ such that $g_{k,n/k}^{r,s} \neq 0$.
\end{enumerate}
\end{lemma}

\begin{proof}
\begin{enumerate}
\item Observe that since $\mu(i)$ is zero whenever $i$ is not square-free, $g_{m,j}^{r,s} = g_{m,p_1 \cdots p_k}^{r,s}$, hence we may assume $t_1 = \ldots = t_k = 1$. In the following computation the index $i$ will be written as ${p_1}^{\beta_1} \cdots {p_k}^{\beta_k}$ where $\beta_1,\ldots,\beta_k \in \{0,1\}$.
\begin{eqnarray}
g_{m,j}^{r,s} & = & m^s \sum_{i|j} \frac{\mu(i) i^s}{\varphi(mi)^r} = m^s \sum_{i|j} \frac{\mu(i) i^s}{\varphi({m'}{p_1}^{\alpha_1} \cdots {p_l}^{\alpha_l} \cdot i)^r} \nonumber \\ & = & \frac{m^s}{\varphi(m')^r} \sum_{i|j} \frac{\mu(i) i^s}{\varphi({p_1}^{\alpha_1} \cdots {p_l}^{\alpha_l} \cdot i)^r} \nonumber \\ & = & \frac{m^s}{\varphi(m')^r} \sum_{v \in \{1,\ldots,k\}, \beta_v \in \{0,1\}} \frac{\mu({p_1}^{\beta_1} \cdots {p_k}^{\beta_k}) {p_1}^{s \beta_1} \cdots {p_k}^{s \beta_k}}{\varphi({p_1}^{\alpha_1+\beta_1} \cdots {p_l}^{\alpha_l+\beta_l} {p_{l+1}}^{\beta_{l+1}} \cdots {p_k}^{\beta_k})^r} \nonumber \\ & = & \frac{m^s}{\varphi(m')^r} \sum_{v \in \{1,\ldots,k\}, \beta_v \in \{0,1\}} \frac{(-p_1^s)^{\beta_1} \cdots (-p_k^s)^{\beta_k}}{\prod_{t=1}^l (p_t-1)^r{p_t}^{r\alpha_t+r\beta_t-r} \prod_{t=l+1}^k (p_t-1)^{r\beta_t}} \nonumber \\ & = & \frac{m^s}{\varphi(m)^r} \sum_{v \in \{1,\ldots,k\}, \beta_v \in \{0,1\}} \frac{(-p_1^s)^{\beta_1} \cdots (-p_k^s)^{\beta_k}}{{p_1}^{r\beta_1} \cdots {p_l}^{r\beta_l} (p_{l+1}-1)^{r\beta_{l+1}} \cdots (p_k-1)^{r\beta_k}} \nonumber \\ & = & \frac{m^s}{\varphi(m)^r} \cdot \prod_{t=1}^l \left( 1-{p_t}^{s-r} \right) \cdot \prod_{t=l+1}^k \left(1-\frac{p_t^s}{(p_t-1)^r} \right). \nonumber
\end{eqnarray}
Since $s \leq r$ and $s \leq 0$, $g_{m,j}^{r,s}$ is a non-negative number, and it is zero if and only if either $p_t^{s-r}=1$ for some $t \in \{1,\ldots,l\}$, i.e. $l \geq 1$ and $s=r$, or $p_t^s=(p_t-1)^r$ for some $t \in \{l+1,\ldots,k\}$, i.e. $s=0$ and either $r=0$ or $p_t=2$ for some $t \in \{l+1,\ldots,k\}$, in other words $j$ is even and $m$ is odd.

\item We have
\begin{eqnarray}
\sum_{m|n} \frac{m^s}{\varphi(m)^r} A(m) & = & \sum_{m|n} \frac{m^s}{\varphi(m)^r} \sum_{k|m} B(k) \mu(m/k) = \sum_{k|m|n} \frac{m^s B(k) \mu(m/k)}{\varphi(m)^r} \nonumber \\ & = & \sum_{k|n,i|n/k} \frac{(ik)^s \mu(i)}{\varphi(ik)^r} B(k) = \sum_{k|n} \left( \sum_{i|n/k} \frac{\mu(i) (ik)^s}{\varphi(ik)^r} \right) B(k) \nonumber \\ & = & \sum_{k|n} g_{k,n/k}^{r,s} B(k). \nonumber
\end{eqnarray}

\item Since by point (1) $g_{k,n/k}^{r,s} \geq 0$ for all $k|n$, applying point (2) to $A,B$ and to $\alpha,\beta$ we obtain $$\sum_{m|n} \frac{m^s}{\varphi(m)^r} A(m) = \sum_{k|n} g_{k,n/k}^{r,s} B(k) \geq \sum_{k|n} g_{k,n/k}^{r,s} \beta(k) = \sum_{m|n} \frac{m^s}{\varphi(m)^r} \alpha(m).$$The statement about equality follows easily. This concludes the proof.
\end{enumerate}
\end{proof}

For $m$ a divisor of $|G|$ set $$c_m := |\{C \leq G\ :\ C\ \text{cyclic},\ |C|=m\}|,$$that is, the number of cyclic subgroups of $G$ of order $m$. Observe that $G$ has exactly $\varphi(m)c_m$ elements of order $m$. It follows that $G$ has exactly $\sum_{k|m} c_k \varphi(k)$ elements of order a divisor of $m$, that is, elements $x$ with the property that $x^m = 1$.

The following is a fundamental fact we will use in an essential way.
\begin{teor}[Frobenius \cite{frob} \cite{frobisaacsrob}] \label{frob}
Let $m$ be a divisor of $|G|$ and let $T_m$ be the set of elements $x \in G$ such that $x^m=1$. Then $m$ divides $|T_m|$.
\end{teor}
Thus we can write $$|T_m| = \sum_{k|m} c_k \varphi(k) = m f(m) \hspace{1cm} (\ast)$$ where $f(m) \geq 1$ is an integer depending on $m$ and $G$.

\ 

Let $r,s$ be two real numbers, let $G$ be a finite group of order divisible by $n$ and let $$R_{G,n}(r,s) := \sum_{x \in G, o(x)|n} \frac{o(x)^s}{\varphi(o(x))^r}, \hspace{1cm} T_{G,n}(r,s) := R_{G,n}(r,s) - R_{C_{|G|},n}(r,s).$$Set $R_G(r,s) := R_{G,|G|}(r,s)$ and $T_G(r,s) := T_{G,|G|}(r,s)$. Note that $T_G(0,0)=0$. We proceed with the proof of Theorem \ref{mainth}.

\subsection{Case $s<r$, $s \leq 0$} \label{case2}

\ 

We prove that if $s < r$ and $s \leq 0$ then $T_{G,n}(r,s) \geq 0$ with equality if and only if $G$ contains a unique cyclic subgroup of order $m$, for every divisor $m$ of $n$.

\ 

Observe that $$R_{G,n}(r,s) = \sum_{m|n} \frac{m^s}{\varphi(m)^{r-1}} c_m.$$ Apply Lemma \ref{nov} to $A(n) = c_n \varphi(n)$, $B(n) = n f(n)$, $\alpha(n) = \varphi(n)$, $\beta(n) = n$. Since $s < r$, if $s < 0$ then $T_{G,n}(r,s) \geq 0$ and equality holds if and only if $f(k) = 1$ for all divisors $k$ of $n$, i.e. $G$ has a unique cyclic subgroup of order $k$ for all divisors $k$ of $n$.

Now suppose that $s=0$. Then following the above argument we find that $g_{k,n/k}^{r,s} \geq 0$ with equality if and only if $p_t=2$ for some $t \in \{l+1,\ldots,h\}$. In other words, equality holds if and only if whenever $k$ is a divisor of $n$ such that either $k$ is even or $n/k$ is odd, $f(k)=1$. We prove that $f(k)=1$ for all divisors $k$ of $n$, from which it follows that if $m$ is a divisor of $n$ then $c_m = \frac{1}{\varphi(m)} \sum_{k|m} k f(k) \mu(m/k) = \frac{1}{\varphi(m)} \sum_{k|m} k \mu(m/k) = 1$, i.e. $G$ has a unique cyclic subgroup of order $m$. So let $k$ be a divisor of $n$. If $k$ is even then $f(k)=1$, and the same is true if $n/k$ is odd, so now suppose that $k$ is odd and $n/k$ is even. In particular $2k$ divides $n$ and is even, hence $f(2k)=1$. We prove that $f(k)=1$. If by contradiction $f(k) \neq 1$ then since $f(k)$ is a positive integer (by Frobenius Theorem), $f(k) \geq 2$ hence $kf(k) \geq 2k$, that is, there are at least $2k$ elements $x \in G$ such that $x^{k}=1$. But these elements also verify $x^{2k}=1$, and $|T_{2k}| = 2kf(2k) = 2k$, i.e. there are exactly $2k$ elements $x$ in $G$ verifying $x^{2k}=1$. This means that every $x \in G$ such that $x^{2k}=1$ verifies $x^k=1$. Now since $n/k$ is even $G$ has an element $x$ of order $2$, and since $k$ is odd $x^k=x$; on the other hand $x^{2k} = x^2 = 1$, a contradiction. Conversely, if $c_m = 1$ for all divisor $m$ of $n$ then if $m$ divides $n$, $f(m) = \frac{1}{m} \sum_{k|m} c_k \varphi(k) = \frac{1}{m} \sum_{k|m} \varphi(k) = 1$.

\subsection{Case $s = r < 0$} \label{case3}

\ 

We prove that if $s=r < 0$ then $T_{G,n}(r,s) \geq 0$ with equality if and only if $G$ contains a unique subgroup of order $n$ and such subgroup is nilpotent.

\ 

As above, using Lemma \ref{nov}, we find $T_{G,n}(r,s) \geq 0$, and equality holds if and only if $f(k) = 1$ whenever $k$ is a divisor of $n$ such that $k$ and $n/k$ are coprime. Applying this to the case when $k$ is a prime power we find that writing $n = p_1^{a_1} \cdots p_u^{a_u}$, $G$ has a unique subgroup of order $p_i^{a_i}$ (which must then be normal in $G$) for $i = 1,\ldots,u$, thus the product of such subgroups is the unique subgroup of $G$ of order $n$, and it is nilpotent. Conversely, if $G$ has a unique subgroup $N$ of order $n$ and $N$ is nilpotent then every Sylow subgroup of $N$ is the unique subgroup of $G$ of its size. Indeed if $P$ is a Sylow subgroup of $N$ and $N = PH$ with $(|P|,|H|)=1$ and $Q$ is some subgroup of $G$ such that $|P|=|Q|$ then $HQ$ is a subgroup of $G$ of order $n$ (because $H \unlhd G$ being the Sylow subgroups of $N$ normal in $G$), hence $HQ=N$ and this implies $Q=P$. Hence if $k$ is a divisor of $n$ such that $(k,n/k)=1$ then $f(k)=1$.

\subsection{Case $r \leq s-1$ and $s \geq 1$} \label{case4}

\ 

We prove that if $r \leq s-1$ and $s \geq 1$ then $T_{G}(r,s) \leq 0$ with equality if and only if $G$ is cyclic.

\ 

The following arguments are the natural generalization of the arguments in \cite{isaacs}. For $X$ a subset of $G$ define $\psi(X) := R_X(r,s) := \sum_{x \in X} \frac{o(x)^s}{\varphi(o(x))^r}$. The following is Lemma C in \cite{isaacs}.

\begin{lemma} \label{largep}
Let $p$ be the largest prime divisor of the integer $n > 1$. Then $\varphi(n) \geq n/p$.
\end{lemma}

\begin{lemma} \label{nphi}
Let $s \geq r$ be two real numbers. Let $m \neq 1$ be a positive divisor of $n$. Then $\frac{n^s}{\varphi(n)^r} \geq \frac{m^s}{\varphi(m)^r}$, and equality holds if and only if one of the following occurs.
\begin{itemize}
\item $s > r$ and $m=n$.
\item $s = r$ and each prime divisor of $n$ divides $m$.
\end{itemize}
\end{lemma}

\begin{proof}
Since the function $n \mapsto f(n) = \frac{n^s}{\varphi(n)^r}$ is multiplicative we may assume that $n = p^a$ and $m = p^b$ with $p$ a prime and $1 \leq b \leq a$. The inequality $f(n) \geq f(m)$ becomes $(p^s)^{a-b} \geq (p^r)^{a-b}$, i.e. $p^{(s-r)(a-b)} \geq 1$ which follows from $a \geq b$, $s \geq r$. If equality holds and $s > r$ we find $a=b$.
\end{proof}

\begin{lemma} \label{cosets}
Let $P$ be a cyclic normal Sylow $p$-subgroup of $G$. Let $x \in G$ and assume that the coset $Px$ has order $m$ as an element of $G/P$. Suppose $s > r$. Then $\psi(Px) \leq  \frac{m^s}{\varphi(m)^r} \psi(P)$ with equality if and only if $x$ centralizes $P$.
\end{lemma}

\begin{proof}
The case $m=1$ is clear, so now assume $m > 1$. Since $m$ divides $o(x)$ we can write $o(x) = mq$ for some integer $q$. Then $q = o(x^m)$ and as $x^m \in P$, we see that $q$ is a power of $p$. But $m$ divides $|G/P|$, which is not divisible by $p$, so $q$ and $m$ are coprime, and there exists an integer $n$ such that $qn \equiv 1 \mod m$. Now $o(x^q) = m$ and we write $y=(x^q)^n$ so that $o(y) = m$ because $n$ is coprime to $m$. Also $Py = Px^{qn} = (Px)^{qn} = Px$ so since $P$ is abelian, $y$ centralizes $P$ if and only if $x$ centralizes $P$. We can thus replace $x$ by $y$ and assume that $o(x) = m$.

\ 

Every element of $Px$ has the form $ux$ for some element $u \in P$, and we argue that $\frac{o(ux)^s}{\varphi(o(ux))^r} \leq \frac{m^s}{\varphi(m)^r} \frac{o(u)^s}{\varphi(o(u))^r}$ with equality if and only if $x$ centralizes $u$. Since $P$ is cyclic, $\langle u \rangle$ is characteristic in $P$, and thus $\langle u \rangle \unlhd G$ and $\langle u \rangle \langle x \rangle$ is a subgroup. Now $ux$ is an element of $\langle u \rangle \langle x \rangle$, so $o(ux)$ divides $|\langle u \rangle \langle x \rangle| = m o(u)$, and if equality holds then $\langle u \rangle \langle x \rangle$ is cyclic hence $x$ centralizes $u$, and conversely if $x$ centralizes $u$ then since $o(x)$ and $o(u)$ are coprime, $o(ux) = o(u)o(x) = mo(u)$. Since $s > r$ it follows from Lemma \ref{nphi} that $\frac{o(ux)^s}{\varphi(o(ux))^r} \leq \frac{m^s o(u)^s}{\varphi(m o(u))^r} = \frac{m^s}{\varphi(m)^r} \frac{o(u)^s}{\varphi(o(u))^r}$ with equality if and only if $x$ centralizes $u$. Now
\begin{eqnarray}
\psi(Px) & = & \sum_{u \in P} \frac{o(ux)^s}{\varphi(o(ux))^r} \leq \sum_{u \in P} \frac{m^s}{\varphi(m)^r} \frac{o(u)^s}{\varphi(o(u))^r} \nonumber \\ & = & \frac{m^s}{\varphi(m)^r} \sum_{u \in P} \frac{o(u)^s}{\varphi(o(u))^r} = \frac{m^s}{\varphi(m)^r} \psi(P). \nonumber
\end{eqnarray}
Moreover equality holds if and only if $x$ centralizes $u$ for every $u \in P$, i.e. $x$ centralizes $P$.
\end{proof}

\begin{cor} \label{psigp}
Let $P$ be a cyclic normal Sylow $p$-subgroup of $G$. Then $\psi(G) \leq \psi(P) \psi(G/P)$ with equality if and only if $P$ is central in $G$.
\end{cor}

\begin{proof}
Write $o(Px)$ to denote the order of a coset $Px$ viewed as an element of $G/P$. Applying Lemma \ref{cosets} to each coset of $P$ in $G$ we have
\begin{eqnarray}
\psi(G) & = & \sum_{Px \in G/P} \psi(Px) \leq \sum_{Px \in G/P} \frac{o(Px)^s}{\varphi(o(Px))^r} \psi(P) \nonumber \\ & = & \psi(P) \sum_{Px \in G/P} \frac{o(Px)^s}{\varphi(o(Px))^r} = \psi(P) \psi(G/P). \nonumber
\end{eqnarray}
By Lemma \ref{cosets} equality holds if and only if every element $x \in G$ centralizes $P$, i.e. $P$ is central in $G$.
\end{proof}

We show that if $T_G(r,s) \geq 0$ then $G$ is cyclic, and we do it by induction on $|G|$. This is a straightforward generalization of the argument used in \cite{isaacs} (proof of the main theorem). Assume $T_G(r,s) \geq 0$, i.e. $R_G(r,s) \geq R_{C_n}(r,s)$, which we can write as $\psi(G) \geq \psi(C_n)$ where $n = |G|$. Then averaging on $G$ and using the fact that $C_n$ has $\varphi(n)$ elements of order $n$ and Lemma \ref{largep}, where $p$ is the largest prime divisor of $n$, we find
\begin{eqnarray}
\frac{1}{|G|} \sum_{x \in G} \frac{o(x)^s}{\varphi(o(x))^r} & \geq & \frac{1}{n} \sum_{x \in C_n} \frac{o(x)^s}{\varphi(o(x))^r} > \frac{1}{n} \frac{n^s}{\varphi(n)^r} \varphi(n) \nonumber \\ & \geq & \frac{1}{n} \frac{n^s}{\varphi(n)^r} \frac{n}{p} \geq \frac{n^s}{\varphi(n)^r} \frac{1}{p}. \nonumber
\end{eqnarray}
The strict inequality comes from the fact that we did not count the contribution of the identity element of $C_n$. This implies that there exists at least one element $x \in G$ ``not below the average'', i.e. such that $$\frac{o(x)^s}{\varphi(o(x))^r} > \frac{n^s}{\varphi(n)^r} \frac{1}{p}.$$ Let $m = o(x)$. Since $s \geq 1$ we have $(n/m)^{s-1} \geq (\varphi(n)/\varphi(m))^{s-1}$ by Lemma \ref{nphi}, and since $s-1 \geq r$ we have $(\varphi(n)/\varphi(m))^{s-1} \geq (\varphi(n)/\varphi(m))^r$. We deduce that $$|G:\langle x \rangle| = \frac{n}{m} \leq \frac{n}{m} \cdot \frac{(n/m)^{s-1}}{(\varphi(n)/\varphi(m))^r} = \frac{(n/m)^s}{(\varphi(n)/\varphi(m))^r} < p.$$Hence $p$ does not divide $|G:\langle x \rangle|$, in other words $\langle x \rangle$ contains a Sylow $p$-subgroup $P$ of $G$. Moreover $\langle x \rangle \subseteq N_G(P)$ hence $|G:N_G(P)| < p$ and it follows from the Sylow theorem that $|G:N_G(P)|=1$, i.e. $P \unlhd G$. Corollary \ref{psigp} then implies that $\psi(G) \leq \psi(P) \psi(G/P)$ with equality if and only if $P$ is central in $G$. Let $Q$ be the Sylow $p$-subgroup of $C = C_n$, so that $P \cong Q$ and $\psi(P) = \psi(Q)$. We have $$\psi(P) \psi(G/P) \geq \psi(G) \geq \psi(C) = \psi(Q) \psi(C/Q) = \psi(P) \psi(C/Q)$$hence $\psi(G/P) \geq \psi(C/Q)$. By the induction hypothesis we deduce that $G/P$ is cyclic. Then $G/P \cong C/Q$ and hence $\psi(G/P) = \psi(C/Q)$. Thus equality holds in the above chain of inequalities and thus $P$ is central in $G$.

Since $P$ is central and $G/P$ is cyclic, it follows that $G$ is abelian, and as $P$ is a Sylow subgroup of $G$, we know that we can write $G = P \times B$ where $B \cong G/P$ is cyclic. Thus $G$ is a direct product of cyclic groups of coprime orders, and thus $G$ is cyclic, as required.

\subsection{Nilpotent case} \label{case5}

\ 

We prove that if $G$ is nilpotent and non-cyclic then the sign of $T_{G}(r,s)$ equals the sign of $r-s$.

\ 

By Lemma \ref{nov} applied to $A(n) = c_n \varphi(n)$, $B(n) = n f(n)$, we have $R_G(r,s) = \sum_{k|n} g^{r,s}_{k,n/k} kf(k)$, hence $$T_G(r,s) = \sum_{k|n} g^{r,s}_{k,n/k} k(f(k)-1).$$ Suppose first that $|G| = p^d$ with $p$ a prime. We then have
\begin{align*}
T_G(r,s) & = \sum_{i=0}^d g^{r,s}_{p^i,p^{d-i}} p^i(f(p^i)-1) = \sum_{i=1}^{d-1} g^{r,s}_{p^i,p^{d-i}} p^i(f(p^i)-1),
\end{align*}
where we used that $f(1) = 1 = f(p^d)$. Now, for $1 \leq i \leq d$ we have $g^{r,s}_{p^i,p^{d-i}} = \frac{p^{is}}{\varphi(p^i)^r}(1-p^{s-r})$. Therefore $g^{r,s}_{p^i,p^{d-i}} < 0$ if and only if $s > r$, $g^{r,s}_{p^i,p^{d-i}} > 0$ if and only if $s < r$, and $g^{r,s}_{p^i,p^{d-i}} = 0$ if and only if $s = r$. Since $G$ is non-cyclic, $f(p^i) \neq 1$ for some $1 \leq i \leq d$ hence $T_G(r,s) < 0$ if $s > r$, $T_G(r,s) = 0$ if $r = s$ and $T_G(r,s) > 0$ if $s < r$.

Now assume that $G$ is any nilpotent non-cyclic group, and write $G = P_1 \times \cdots \times P_t$ as direct product of its Sylow subgroups. Note that $R_G(r,s)$ is multiplicative in the sense that if $A,B$ are groups of coprime orders then $R_{A \times B}(r,s) = R_A(r,s) R_B(r,s)$. Write $|P_i| = p_i^{a_i}$ and $n = |G| = p_1^{a_1} \cdots p_t^{a_t}$. If $r < s$ then, since $G$ is non-cyclic, by our above discussion of $p$-groups $R_{P_i}(r,s) < R_{C_{p_i^{a_i}}}(r,s)$ for some $i \in \{1,\ldots,t\}$ hence taking the product we find $R_G(r,s) < R_{C_n}(r,s)$. Similarly if $r > s$ then $R_G(r,s) > R_{C_n}(r,s)$ and if $r=s$ then $R_G(r,s) = R_{C_n}(r,s)$.

\section{Some observations}

Consider now the case $(r,s) = (1,0)$. In this case the result looks as follows. Let $d(n)$ denote the number of positive divisors of the integer $n$.

\begin{cor} \label{cormotiv}
Let $n$ be a divisor of $|G|$. Then $G$ has at least $d(n)$ cyclic subgroups of order a divisor of $n$. Moreover the following are equivalent.
\begin{enumerate}
\item For all divisors $m$ of $n$, $G$ has exactly $m$ elements $x$ such that $x^m=1$.
\item $G$ has exactly $d(n)$ cyclic subgroups of order a divisor of $n$.
\item The subgroup generated by the cyclic subgroups of $G$ of order a divisor of $n$ is cyclic and has order $n$.
\end{enumerate}
\end{cor}

\begin{proof}
We prove that the number of cyclic subgroups of a finite group $G$ whose order divides $n$ equals $\sum_{x \in G, o(x)|n} \frac{1}{\varphi(o(x))}$. Let $\langle x_1 \rangle, \ldots, \langle x_k \rangle$ be the distinct cyclic subgroups of $G$ of order a divisor of $n$. Then $\langle x_i \rangle$ contains $\varphi(o(x_i))$ elements of order $o(x_i)$. For $x, y \in G$ write $x \sim y$ if $x,y$ generate the same cyclic subgroup of $G$. Then $\sim$ is an equivalence relation in $\{x \in G\ :\ o(x)|n\}$ hence $$\sum_{x \in G,o(x)|n} \frac{1}{\varphi(o(x))} =  \sum_{i=1}^k \sum_{x \sim x_i} \frac{1}{\varphi(o(x))} = \sum_{i=1}^k \frac{\varphi(o(x_i))}{\varphi(o(x_i))} = k.$$Now the result follows from Theorem \ref{mainth} choosing $(r,s) = (1,0)$.
\end{proof}

\includegraphics{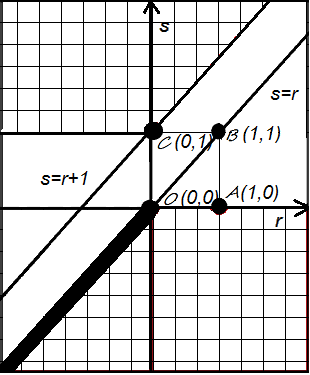}

De Medts and Tarnauceanu in \cite{medts2} prove that if $G$ is nilpotent then $T_G(1,1)=0$, and they conjecture that the converse holds, namely if $T_G(1,1)=0$ then $G$ is nilpotent. Also, they conjecture in \cite{medts} that $T_G(1,1) \leq 0$ for all finite groups $G$.

\begin{question}
Does the equation $T_G(r,s)=0$ detect solvability of $G$ for some $(r,s)$? (This is a version of a question of Thompson, cf. \cite{medts2}).
\end{question}

\begin{question}
What structural properties of $G$ can be detected by the equation $T_G(r,s)=0$ (for fixed $r,s$)?
\end{question}

\begin{question}
Is the set of pairs $(r,s)$ for which cyclic groups are detected by the equation $T_G(r,s)=0$ dense in $\mathbb{R}^2$?
\end{question}

In the picture above the point $O = (0,0)$ does not detect anything, $A = (1,0)$ (which corresponds to the number of cyclic subgroups of $G$) and $C = (0,1)$ (which corresponds to the sum of the orders of the elements of $G$) detect cyclicity and whether $B = (1,1)$ detects nilpotency is the question of De Medts and Tarnauceanu. The checked zone detects cyclicity and the thick line detects nilpotency.

\section{Proof of Theorem \ref{thprod}} \label{sprod}

Let $G$ be a finite group of order $n$. For a divisor $d$ of $n$ let $A(d)$ be the number of elements of $G$ of order $d$. Then $P_G := \prod_{x \in G} o(x) = \prod_{d|n} d^{A(d)}$ hence $\log(P_G) = \sum_{d|n} A(d) \log(d)$. Let $B(m)$ be the size of $\{x \in G\ :\ x^m=1\} \subseteq G$ for $m|n$. Then $\sum_{d|m} A(d) = B(m)$. We compute

\begin{eqnarray}
\sum_{m|n} \log(m) A(m) & = & \sum_{m|n} \log(m) \sum_{k|m} B(k) \mu(m/k) = \sum_{k|m|n} \log(m) B(k) \mu(m/k) \nonumber \\ & = & \sum_{k|n, i|n/k} \log(ik) \mu(i) B(k) = \sum_{k|n} \left( \sum_{i|n/k} \mu(i) \log(ik) \right) B(k) \nonumber \\ & = & \sum_{k|n} g_{k,n/k} B(k) \nonumber
\end{eqnarray}
where $g_{m,j} = \sum_{i|j} \mu(i) \log(mi)$. Clearly $g_{m,1} = \log(m)$ and if $j > 1$ then $g_{m,j} = \sum_{i|j} \mu(i) \log(i)$. We now compute $g_{m,j}$ for $j > 1$. If $j$ is a power of a prime $p$ then $g_{m,j} = -\log(p)$. Now assume this is not the case. Let $P$ be the set of prime divisors of $j$, so that $|P| > 1$. We prove that $g_{m,j} = 0$.
\begin{eqnarray}
g_{m,j} & = & \sum_{i|j} \mu(i) \log(i) = \sum_{I \subseteq P} (-1)^{|I|} \sum_{p \in I} \log(p) \nonumber \\ & = & \sum_{p \in P} \sum_{p \in I \subseteq P} (-1)^{|I|} \log(p) = - \sum_{p \in P} \log(p) \sum_{I \subseteq P-\{p\}} (-1)^{|I|} = 0. \nonumber
\end{eqnarray}
We used that $P-\{p\} \neq \emptyset$.

In conclusion, writing $n = p_1^{c_1} \cdots p_t^{c_t}$ and $B_i := \sum_{j=1}^{c_i} B(n/p_i^j)$ we have that $\log(P_G)$ equals
\begin{align*}
\sum_{m|n} \log(m) A(m) = & \sum_{k|n} g_{k,n/k} B(k) = \log(n) B(n) - \sum_{i=1}^t \log(p_i) \sum_{j=1}^{c_i} B(n/p_i^{j}) \\ = & \log \left( \frac{n^{B(n)}}{p_1^{\sum_{j=1}^{c_1} B(n/p_1^{j})} \cdots p_t^{\sum_{j=1}^{c_t} B(n/p_t^{j})}} \right) = \log \left( \frac{n^{B(n)}}{p_1^{B_1} \cdots p_t^{B_t}} \right).
\end{align*}
Since $B(n)=n$ we obtain that $$P_G = \frac{n^{n}}{p_1^{B_1} \cdots p_t^{B_t}}.$$
Since $B(m) \geq m$ for all $m|n$ (by Frobenius Theorem), we obtain $$P_G = \frac{n^{n}}{p_1^{B_1} \cdots p_t^{B_t}} \leq \frac{n^{n}}{p_1^{\sum_{j=1}^{c_1} n/p_1^{j}} \cdots p_t^{\sum_{j=1}^{c_t} n/p_t^{j}}} = P_{C_n}.$$Suppose now that equality holds. Then $B(n/p_i^{j}) = n/p_i^{j}$ for $i = 1,\ldots,t$ and $j = 1,\ldots,c_i$. In particular if $p|n$ then $B(n/p) = n/p$ hence there are $n-n/p > 1$ elements of $G$ of order not dividing $n/p$, in particular the Sylow $p$-subgroups of $G$ are cyclic. In particular if $G$ is a $p$-group we are done, so now assume that $n$ is divisible by at least two distinct primes. Let $P$ be a Sylow $p$-subgroup of $G$. We prove that $P$ is central in $G$. To do this it is enough to show that for all prime $q$ dividing $n$ there is a Sylow $q$-subgroup $Q$ of $G$ which centralizes $P$. If $p=q$ choose $Q=P$. Now suppose $q \neq p$. Write $c_p$, $c_q$ for the positive integers such that $p^{c_p}$, $q^{c_q}$ are the largest powers of $p,q$ dividing $n$, respectively. It is enough to prove that $G$ contains elements of order $p^{c_p} q^{c_q}$. There are $n/p$ elements of order dividing $n/p$ and $n/q$ elements of order dividing $n/q$, hence there are at least $n-n/p-n/q$ elements of order not dividing $n/p$ and not dividing $n/q$, in other words there are at least $n-n/p-n/q$ elements of order divisible by $p^{c_p} q^{c_q}$, a suitable power of which has order $p^{c_p} q^{c_q}$. We only need to make sure that $n-n/p-n/q > 0$. We have $n-n/p-n/q \geq n-n/2-n/3 = n/6 > 0$.

Since the Sylow subgroups of $G$ are central and cyclic, $G$ is cyclic. The proof is completed.

\section{An example} \label{examples}

Let $\gamma$ be a fixed positive integer. In this section we construct infinitely many finite groups $G$ with exactly $d(|G|) + \gamma$ cyclic subgroups. This is done by applying Proposition \ref{subcont} with $m=3$, $\beta = 5^{\gamma-1}$.

\ 

Let $\alpha, \beta, m, u$ be positive integers with $\alpha = 2^u \beta$ (since $u \geq 1$, $\alpha$ is even), $\beta$ odd and $(\alpha,m) = 1$. Let $G := C_m \rtimes C_{\alpha}$ where the action of $C_{\alpha} = \langle x \rangle$ on $C_m = \langle a \rangle$ is given by inversion: $a^x = x^{-1}ax = a^{-1}$.

\begin{prop} \label{subcont}
The number of cyclic subgroups of $G$ is $d(|G|) + d(\beta) (m-d(m))$.
\end{prop}

For the record, $$T_G(r,s) = \frac{2^{us}}{\varphi(2^u)^{r-1}} (m-\sigma_{r,s}(m)) \sigma_{r,s}(\beta)$$ where $\sigma_{r,s}(x) := \sum_{i|x} \frac{i^s}{\varphi(i)^{r-1}}$.

\begin{proof}
We will first find the cyclic subgroups generated by elements of the form $a^r x^w$ where $w$ divides $\beta$ and then we will look at the elements of the form $a^r x^{2t}$.

\ 

Let $r \in \{0,\ldots,m-1\}$ and let $w$ be a divisor of $\beta$. We compute $o(a^r x^w)$. Since $\beta$ is odd, $w$ is odd, hence $a^{x^w} = a^{-1}$ hence $(a^r x^w)^2 = x^{2w}$. This implies that $(a^r x^w)^n$ equals $x^{nw}$ if $n$ is even and $a^r x^{nw}$ if $n$ is odd, implying that $o(a^r x^w) = o(x^w) = \alpha/w$, which is divisible by $2^u$. Moreover clearly $a^r$ is determined by $\langle a^r x^w \rangle$ (all the elements of $\langle a^r x^w \rangle$ outside $\langle x \rangle$ are of the form $a^r x^{nw}$) and $x^w$ is determined by $\langle a^r x^w \rangle$ because $o(a^r x^w) = \alpha/w$. This means that different pairs $(r,w)$ give rise to distinct cyclic subgroups $\langle a^r x^w \rangle$. Hence there are $m d(\beta)$ such subgroups. Hence we have found $m d(\beta)$ cyclic subgroups of the form $a^r x^w$ where $w$ is a divisor of $\beta$ (so $w$ is odd).

\ 

Since $a$ commutes with $x^2$ we have $(ax^2)^n = a^n x^{2n}$ hence $o(ax^2) = m \alpha/2$. Moreover the elements of $G$ of the form $a^r x^{2t}$ (i.e. $a^r$ times an even power of $x$) verify $(a^r x^{2t})^{m \alpha/2} = a^{rm} x^{\alpha} = 1$ hence their order divides $m \alpha/2$, on the other hand there are exactly $m \alpha/2$ such elements hence they all belong to $\langle a x^2 \rangle$. This proves that $\langle a x^2 \rangle$ is the only cyclic subgroup of $G$ of order $m \alpha/2$.

\ 

We have found $m d(\beta) + d(m \alpha/2)$ cyclic subgroups of $G$. We have $d(|G|) = d(m \beta 2^u) = (u+1) d(m) d(\beta)$ and $d(m \alpha/2) = d(m) d(\beta) u$, hence $$m d(\beta) + d(m \alpha/2) = m d(\beta) + d(m) d(\beta) u = d(|G|) + m d(\beta) - d(m) d(\beta).$$ Therefore all we are left to show is that the cyclic subgroups we listed are all the cyclic subgroups of $G$.

\ 

We need to show that if $C$ is a cyclic subgroup of $G$ of order divisible by $2^u$ then $C$ is generated by an element of the form $a^r x^w$ with $w$ a divisor of $\beta$. Say $C$ is generated by $a^r x^t$. Then $t$ is odd, indeed if $t$ is even then $a^r$ and $x^t$ commute hence $(a^r x^t)^{m \alpha/2} = 1$, contradicting the fact that $o(a^r x^t)$ is divisible by $2^u$. Now, similarly as above, $a^r x^t$ has order $o(x^t) = \alpha/w$ where $w = (\alpha,t)$ is an odd divisor of $\alpha$, that is, a divisor of $\beta$. Since $\langle x^t \rangle = \langle x^w \rangle$ there is some integer $s$ such that $x^{ts}=x^w$ and $s$ is coprime to $o(x^t)=o(a^rx^t)$. This implies that $\langle a^r x^t \rangle = \langle (a^r x^t)^s \rangle$ and $(a^r x^t)^s$ equals $a^r x^w$ if $s$ is odd, it equals $x^w$ if $s$ is even.
\end{proof}

\section{Acknowledgements}

We are really grateful to Andrea Lucchini and Federico Menegazzo for helpful discussions, comments and suggestions. We are also very grateful to the referee for his/her very careful reading of a previous version of the paper.

\bigskip

{\it Martino Garonzi, Dipartimento di Matematica Pura ed
Applicata, Via Trieste 63, 35121 Padova, Italy.

E-mail address: mgaronzi@math.unipd.it}

\medskip

{\it Massimiliano Patassini

E-mail address: frapmass@gmail.com}

\end{document}